\documentclass[12pt]{amsart}

\usepackage[margin=1.15in]{geometry}

\usepackage{amsmath,amscd,amssymb,amsfonts,latexsym,wasysym, mathrsfs, mathtools,hhline,color, bm}
\usepackage[all, cmtip]{xy}

\usepackage{url}

\definecolor{hot}{RGB}{65,105,225}

\usepackage[pagebackref=true,colorlinks=true, linkcolor=hot ,  citecolor=hot, urlcolor=hot]{hyperref}%make all the references and links clickable

\usepackage{ textcomp }
\usepackage{ tipa }

\usepackage{graphicx,enumerate}

\usepackage{enumitem}
\usepackage[normalem]{ulem}

\theoremstyle{plain}
\newtheorem{theorem}{Theorem}[section]
\newtheorem{proposition}[theorem]{Proposition}

\newtheorem{lm}[theorem]{Lemma}

\newtheorem{corollary}[theorem]{Corollary}

\newtheorem{lemma}[theorem]{Lemma}
\newtheorem{thrm}[theorem]{Theorem}

\theoremstyle{definition}

\newtheorem{definition}[theorem]{Definition}

\newtheorem{remark}[theorem]{Remark}
\newtheorem{notation}[theorem]{Notation}

\newtheorem{ex}[theorem]{Example}
\newtheorem*{ex*}{Example}
\newtheorem{problem}{Problem}

\def\be{\begin{equation}}
\def\ee{\end{equation}}

\def\bt{\begin{thrm}}
\def\et{\end{thrm}}

\def\bc{\begin{cor}}
\def\ec{\end{cor}}

\def\br{\begin{rmk}}
\def\er{\end{rmk}}

\def\bp{\begin{prop}}
\def\ep{\end{prop}}

\def\bl{\begin{lm}}
\def\el{\end{lm}}

\def\bex{\begin{ex}}
\def\eex{\end{ex}}

\def\bd{\begin{defn}}
\def\ed{\end{defn}}

\newcommand\sE{{\mathcal E}}

\newcommand\sU{{\mathcal U}}
\newcommand\sV{{\mathcal V}}

\newcommand\sZ{\mathcal{Z}}

\def\bfw{\mathbf{w}}
\def\bfu{\mathbf{u}}
\def\bfx{\mathbf{x}}

\newcommand\pp{{\mathbb{P}}}

\newcommand\cc{{\mathbb{C}}}

\DeclareMathOperator{\GED}{gEDdeg}                  
\DeclareMathOperator{\UED}{uEDdeg}                  
\DeclareMathOperator{\DED}{EDdefect}                  
\DeclareMathOperator{\Jac}{Jac}

\DeclareMathOperator{\codim}{codim}              % codim
\DeclareMathOperator{\reg}{reg}                  % reg
\DeclareMathOperator{\sing}{Sing}                  
\DeclareMathOperator{\Eu}{Eu}

\DeclareMathOperator{\EDdeg}{EDdeg}
\DeclareMathOperator{\PEDdeg}{EDdeg_{proj}}

\def\bC{\mathbb{C}}
\def\RR{\mathbb{R}}

\def\bP{\mathbb{P}}

\def\lra{\longrightarrow}
\def\bQ{\mathbb{Q}}

\def\cL{\mathcal{L}}

\def\bZ{\mathbb{Z}}

\def\balpha{{\bm\alpha}}
\def\bbeta{{\bm\beta}}
\def\C{\mathbb{C}}

\newcommand{\blue}[1]{{\color{blue}#1}}

\title[UED versus GED]{Defect of Euclidean distance degree}
\author{Laurentiu G. Maxim}
\address{Department of Mathematics,         University of Wisconsin-Madison,  480 Lincoln Drive, Madison WI 53706-1388, USA.}
\email {maxim@math.wisc.edu}\urladdr{https://www.math.wisc.edu/~maxim/}
\author{Jose Israel Rodriguez}
\address{Department of Mathematics,         University of Wisconsin-Madison,  480 Lincoln Drive, Madison WI 53706-1388, USA.}
\email {jose@math.wisc.edu}\urladdr{http://www.math.wisc.edu/~jose/}
\author{Botong Wang}
\address{Department of Mathematics,         University of Wisconsin-Madison,  480 Lincoln Drive, Madison WI 53706-1388, USA.}
\email {wang@math.wisc.edu}\urladdr{http://www.math.wisc.edu/~wang/}

 %optimal solution,    %stationary point,       maximum likelihood,    objective function}


\begin{document}

\date{\today}

\begin{abstract}  
Two well studied invariants of a complex projective variety are the unit Euclidean distance degree and the generic Euclidean distance degree. 
These numbers give a measure of the algebraic complexity for ``nearest" point problems of the algebraic variety.
It is well known that the latter is an upper bound for the former.  While  this bound may be tight,  
many varieties appearing in optimization, engineering, statistics, and data science, have a significant gap between these two numbers. 
We call this difference the defect of the ED degree of an algebraic variety. 
In this paper we compute this defect by classical techniques in Singularity Theory, thereby deriving a new method for computing ED degrees of smooth complex projective varieties. 
\end{abstract}

\maketitle

\section{Introduction}

The unit Euclidean distance degree and the generic Euclidean distance degree are two well-studied invariants which give a measure of the algebraic complexity for ``nearest" point problems of an algebraic variety.
%\magenta{\sout{ In this paper, we study the difference between these two invariants and give a topological interpretation of  this ``defect'' for a smooth complex projective variety. }}

\begin{definition}
Let $X$ be an irreducible closed subvariety of $\cc^n$. 
The
{\it $\bfw$-weighted Euclidean distance degree} of $X$ is the 
 number of complex critical points of 
 \[
 d_{\bfu,\bfw}(\bfx):=\sum_{i=1}^n w_i(x_i-u_i)^2
 \] 
on the smooth locus $X_{\reg}$ of $X$, 
for generic data $\bfu=(u_1,\dots,u_n)$.
 We write this degree as $\EDdeg_\bfw(X)$. 
When $\bfw$ is generic (resp., $\bfw=\bf{1}$, the all ones vector) we call $\EDdeg_\bfw(X)$ the {\it generic ED degree} (resp., {\it unit ED degree}) of $X$ and write this as $\GED(X)$ (resp., $\UED(X)$).

 \end{definition}

%\blue{
The {Euclidean distance degree} was introduced in \cite{DHOST}, and has since been extensively studied in areas like 
computer vision~\cite{PST2017,HL,MRW2018},
biology~\cite{GHRS2016}, 
chemical reaction networks~\cite{AH2018}, 
engineering~\cite{CNAS2014,SHA2015},
numerical algebraic geometry~\cite{Hau2013,MR2018}, 
%statistics
  and data science~\cite{HW18}.
Also of interest are ED-discriminant loci  
\cite{Hor2017,BKSW2018}, which characterize the meaning of ``generic data" in terms of vanishing of polynomials,
and the algebraic degree of other optimization problems \cite{BPT2013, HS2014,NRS2010}.
%}

\medskip

For a projective variety $X$  one defines the  $\bfw$-weighted Euclidean distance degree of $X$ in terms of affine cones. 

\begin{definition}
If X is an irreducible closed subvariety of $\pp^n$, we define the (projective) $\bfw$-weighted Euclidean distance degree of $X$ by
$\EDdeg_\bfw(X) := \EDdeg_\bfw(C(X))$,
where $C(X)$ is the affine cone of $X$ in $\cc^{n+1}$. 
 \end{definition}

It was proved in \cite[Theorem 1.3]{MRWp} (see also \cite[Theorem 8.1]{AH} for the smooth case) that $\UED$ of a projective variety can be computed as an Euler characteristic weighted by a certain constructible function. More precisely, one has the following result:
\begin{theorem}\label{thm1}
Let $X \subset \bP^n$ be an irreducible closed subvariety. Then
\begin{equation}\label{eq1} \UED(X)= (-1)^{\dim (X)} \chi( {\rm Eu}_X \vert_{\pp^n  \setminus (Q \cup H)}),\end{equation}
where ${\rm Eu}_X$ is the local Euler obstruction function on $X$, 
$Q$ is the isotropic quadric $\{(x_0:\dots:x_n)\in\pp^n\mid \sum_{i=0}^nx_i^2=0\}$,
and $H$ is a general hyperplane in $\bP^n$. In particular, if $X$ is smooth, then 
\begin{equation}\label{eq2}\UED(X)= (-1)^{\dim (X)} \chi(X \setminus (Q \cup H)).\end{equation}
\end{theorem}

Moreover, the above result can be extended to the computation of $\EDdeg_\bfw(X)$, for an arbitrary weight $\bfw$ (see Theorem \ref{thm2} below).

 The unit ED degree, $\UED(X)$, is in general difficult to compute even if $X$ is smooth, since the isotropic quadric $Q$ may intersect $X$ non-transversally. 
On the other hand, for generic weight $\bfw$, the quadric $Q_\bfw$ intersects $X$ transversally, and the computation of $\GED(X)$ is more manageable, see e.g., \cite{DHOST}, \cite{OSS}, etc.

\medskip

In this paper, we study the difference $$\DED(X)\coloneqq\GED(X)-\UED(X)$$  which we refer to as the {\it defect} of the Euclidean distance degree. It is known that $\DED(X)$ is non-negative, but for many varieties appearing in optimization, engineering, statistics, and data science, the defect is quite substantial. We give a new topological interpretation of this defect 
%\blue{But we already have a topological interpretation by taking the difference of the right hand side's of \eqref{eq2} and \eqref{eq4} by our previous paper.  
%What is the advantage of our new description in terms of 
%invariants of singularities of $X \cap Q$? }
in terms of invariants of singularities of $X \cap Q$ when $X$ is a smooth irreducible complex projective variety in $\pp^n$.

%There are several advantages to this approach.
Even though both $\GED(X)$ and $\UED(X)$ can be computed by topological invariants, as seen in (\ref{eq2}) and (\ref{eq4}), our new approach provides a direct method for computing Euclidean degree defects. This approach is applied in Section~\ref{sec:example} on very concrete examples. 
In particular, in Example~\ref{ex:illustrative} we show that the ED defect can be computed much easier than computing $\GED(X)$ and $\UED(X)$ individually. 

%in Example \ref{} 
%\blue{
%Insert advantages related to our examples.}

 Our results also recover and give a more conceptual interpretation of a result of Aluffi-Harris \cite{AH}, which was obtained by characteristic class techniques. Before stating the main result of this paper, let us fix some notations. 
 
\begin{notation}\label{not1}
Let $Q=\{(x_0:\ldots:x_n)\in\pp^n\mid \sum_{i=0}^nx_i^2=0\}$ be the isotropic quadric, and let $X \subset \pp^n$ be a smooth irreducible projective variety not contained in $Q$. Let $Z=\sing(X\cap Q)$ be the singular locus of $X\cap Q$, taken as the schematic intersection.
%\footnote{It is possible that $X\cap Q$ with the induced reduced structure is smooth, but the schematic intersection is generically non-reduced. See Example \ref{}}. 
Equivalently, $Z\subset X\cap Q$ is the locus where $X$ intersects $Q$ non-transversally. Let $\mathscr{X}$ be a Whitney stratification of $X\cap Q$, and denote by $\mathscr{X}_0$ the collection of strata contained in~$Z$. 
\end{notation}

In the above notations, our main result can be stated as follows (see Theorem \ref{thm-gen}):

\begin{theorem}[\textnormal{Main Result}]\label{thm-i1} % 
%Let $Q=\{(x_0:\ldots:x_n)\in\pp^n\mid \sum_{i=0}^nx_i^2=0\}$ be the isotropic quadric, and let $X \subset \pp^n$ be a smooth irreducible projective variety not contained in $Q$. 
%Let $Z$ denote the subset of $X\cap Q$ where the intersection is not transverse. 
%Suppose $X\cap Q$  has a Whitney stratification $\mathscr{X}$, and denote by $\mathscr{X}_0$ the collection of strata contained in~$Z$. Then:
Let $X \subset \pp^n$ be a smooth irreducible projective variety not contained in the isotropic quadric $Q$. Then, under Notation \ref{not1},
\begin{equation}\label{i3}
\DED(X)%=\GED(X)-\UED(X)
=\sum_{V \in \mathscr{X}_0} (-1)^{\codim_{X \cap Q} V} \alpha_V \cdot\GED(\bar V)
\end{equation}
with $$\alpha_V=\mu_V-\sum_{\{S \mid V < S\}} \chi_c(L_{V,S}) \cdot \mu_S,$$
where, for any stratum $V\in \mathscr{X}_0$, $$\mu_V=	\chi(\widetilde{H}^*(F_{V};\bQ))$$
is the Euler characteristic of the reduced cohomology of the Milnor fiber $F_{V}$ of the hypersurface $X \cap Q \subset X$ at some point in $V$, and $L_{V,S}$ is the complex link of a pair of distinct strata $(V,S)$ with $V \subset {\bar S}$. %{\color{red}Since we do not define $L_{V, S}$ here, shall we use the Euler obstruction $e_{V, S}$ instead of $\chi_c(L_{V,S})$?}
\end{theorem} 
\begin{remark}
For the precise definition of the Milnor fiber $F_V$, see \cite{Mil68}, and for the complex link $L_{V, S}$, see \cite[Theorem 1.1]{EM} and \cite[Theorem 2.10]{RW}.
\end{remark}

\begin{remark}
By Thom's second isotopy lemma,  the topological type of Milnor fibers is constant along the strata of a Whitney stratification $\mathscr{X}$ of the hypersurface $X \cap Q$ in~$X$.
\end{remark}

As an immediate consequence of Theorem \ref{thm-i1}, we get the following result (see Corollary \ref{cor1}, and compare also with \cite[Corollary 6.3]{AH}):
\begin{corollary}[\textnormal{Isolated Singularities}]\label{cor-i1}  % Isolated singularities
	%Let $X\subset \pp^n$ be a smooth irreducible projective variety, and let 
	%$Q$ denote the isotropic quadric $\{(x_0:\dots:x_n)\in\pp^n\mid \sum_{i=0}^nx_i^2=0\}$.
	%$Q=\{x_0^2+\cdots + x_n^2\}$ be the isotropic quadric in $\pp^n$.
	 Under Notation \ref{not1}, assume that $\sing(X\cap Q)$ consists of  isolated points. Then:
	\begin{equation}
\DED(X)=\sum_{x \in {\rm Sing}(X \cap Q)} \mu_x,
\end{equation}
where $\mu_x$ is the Milnor number of the isolated singularity $x \in {\rm Sing}(X \cap Q)$.
\end{corollary}

Furthermore, if $X \cap Q$ is equisingular along the non-transversal intersection locus $Z$, Theorem \ref{thm-i1} yields the following:
\begin{corollary}[\textnormal{Equisingular singular locus}]
%Let $X\subset \pp^n$ be a smooth irreducible projective variety. 
Under Notation \ref{not1}, assume that $Z=\sing(X\cap Q)$ is connected and $X \cap Q$ is equisingular along $Z$.
%its singular (i.e., non-transversal intersection) locus $Z$. 
Then:
\begin{equation}
\DED(X)=\mu \cdot \GED(Z),
\end{equation}
where $\mu$ is the Milnor number of the isolated transversal singularity at some point of $x$ in $Z$ (i.e., the Milnor number of the isolated hypersurface singularity in a normal slice to $Z$ at $x$).
\end{corollary}
%\blue{Does ``Milnor number of the isolated transversal singularity at some point $p$ of $Z$" mean the following. Take a linear space $\cL$ with codimension $\dim Z$ containing $p$ of $Z$, then the $\mu$ is the Milnor number of the isolated singularity $p$ of $X\cap Q\cap \cL$ ?  }

Theorem \ref{thm-i1} is motivated by the duality conjecture of \cite{OSS}[(3.5)] in structured low-rank approximation, which predicts a formula for the Euclidean distance degree defect of the restriction of (the dual variety of) $X$ to a linear space $\cL$. 
Since intersecting $X$ with a general linear space $\cL$ does not change the multiplicities $\alpha_V$ on the right-hand side of formula (\ref{i3}), we get the following consequence of Theorem ~\ref{thm-i1}:
\begin{corollary}[Intersection with linear space]\label{cor-slice}  %% Slice
With the notations as in Theorem~\ref{thm-i1},
let $\cL$ denote a general linear subspace of $\pp^n$.
Then
\begin{equation}
\DED(X\cap\cL)%=\GED(X)-\UED(X)
=\sum_{V \in \mathscr{X}_0} (-1)^{\codim_{X \cap Q} V} \alpha_V \cdot\GED(\bar V\cap \cL).
\end{equation}
\end{corollary}

The proof of our main Theorem \ref{thm-i1} relies on the theory of vanishing cycles adapted to a pencil of quadrics 
%$Q_{\bfw}=\{w_0x_0^2+\cdots+w_nx_n^2=0\}$ on $X$, % OLD
$Q_{\bfw}=\{(x_0:\dots:x_n)\in\pp^n\mid w_0x_0^2+\cdots+w_nx_n^2=0\}$ on $X$,  %NEW
 see Theorem \ref{th-main}. For a quick introduction to hypersurface singularities and vanishing cycles, the interested reader may consult \cite{Max}[Chapter 10].

\medskip

{\bf Acknowledgements.} L. Maxim is partially supported by the Simons Foundation Collaboration Grant \#567077 and by the
Romanian Ministry of National Education, CNCS-UEFISCDI, grant PN-III-P4-ID-PCE-2016-0030. 
J.~I. Rodriguez is partially supported by the College of Letters and Science, UW-Madison. 
B. Wang is partially supported by the NSF grant DMS-1701305.

%%%%%%%%%%%%%%%%%%%%%%%%%%%%%

\section{Computation of ED defect via vanishing cycles}
In this section, we compute the defect $\DED(X)$ by using standard techniques in Singularity theory, such as Milnor fibers, vanishing cycles and the local Euler obstruction function.

We begin with the observation that the proof of Theorem \ref{thm1}  in \cite{MRWp} applies without change to the context of the $\bfw$-weighted Euclidean distance degree of $X$, if one uses instead the quadric $Q_{\bfw}=\{(x_0:\dots:x_n)\in\pp^n\mid w_0x_0^2+\cdots+w_nx_n^2=0\}$. More precisely, one has the following result.
\begin{theorem}\label{thm2}
Let $X \subset \bP^n$ be an irreducible closed subvariety. Then,
\begin{equation}\label{eq3} 
\EDdeg_\bfw(X)= (-1)^{\dim (X)} \chi( {\rm Eu}_X \vert_{\pp^n \setminus (Q_\bfw \cup H)}),\end{equation}
with $H$ a general hyperplane. 
In particular, if $X$ is smooth one has: 
\begin{equation}\label{eq4}
\EDdeg_\bfw(X)= (-1)^{\dim (X)} \chi(X \setminus (Q_\bfw \cup H)).
\end{equation}
\end{theorem}

\medskip

From now on we assume in this section that $X$ is a smooth irreducible complex projective variety in $\pp^n$, which is not contained in the isotropic quadric $Q$. Then $X \cap Q_\bfw$ yields a pencil of hypersurfaces $X_s=\{f_s=0\}_{s \in \pp^1}$ on $X$, where $s=[s_0: s_1]$ and
\[
f_s=s_0(x_0^2+\cdots+x_n^2)+s_1(w_0x_0^2+\cdots+w_nx_n^2).
\]
The generic member of the pencil is $X_{\infty}:=X \cap Q_\bfw$ for $\bfw$ generic, and the singular member is $X_0:=X \cap Q$ for $Q=Q_{\bf{1}}$ the isotropic quadric. Moreover, the generic $\bfw$ can be chosen so that the generic member $X_{\infty}=X \cap Q_\bfw$ of the pencil is a smooth hypersurface in $X$ (since in this case $X$ and $Q_\bfw$ intersect transversally), which is transversal to the strata of a Whitney stratification of $X_0=X \cap Q$. 

Consider the incidence variety of the pencil, that is, $$\widetilde{X}:=\{(s,x)\in \pp^1 \times X \mid x \in X_s\},$$ which is just the blowup of $X$ along the base locus $X_0 \cap X_{\infty}$ of the pencil. Let $\pi:\widetilde{X} \to \pp^1$ be the projection map, hence $X_s=\pi^{-1}(s)$ for any $s \in \pp^1$. Let 
$$f:=\frac{f_0}{f_{\infty}} : X \setminus X_{\infty} \subset \widetilde{X} \lra \bC$$
with $f^{-1}(0)=X_0 \setminus X_{\infty}$. 

With the above assumptions and notations, we can now prove the following result:
\begin{theorem}\label{th-main} Let $X \subset \pp^n$ be a smooth irreducible complex projective variety, and let $\bfw$ be a generic weight. Then:
	 \begin{equation}\label{th-eq}
	%\UED(X) - \GED(X)
	-\DED(X)
	= 	(-1)^{\dim (X)} \chi(\varphi_{f} (1_{X \setminus Q_\bfw})|_{X\setminus(Q_\bfw \cup H)})
	 \end{equation}
\end{theorem}

\begin{proof}
	In the above notations and for generic $\bfw$, additivity properties of the Euler characteristic for complex algebraic varieties, together with formulae (\ref{eq2}) and (\ref{eq4}) yield (here we choose a hyperplane $H$ which is generic in both situations):
\begin{equation}\label{eq5}
\begin{split}
\UED(X) - \GED(X) &=(-1)^{\dim (X)}	
\left[ \chi(X \setminus (Q \cup H)) - \chi(X \setminus (Q_\bfw \cup H))\right] \\
&= (-1)^{\dim (X)}	
\left[ \chi(X \cap (Q_\bfw \cup H)) - \chi(X \cap (Q \cup H))\right] \\
&= (-1)^{\dim (X)}	
\left[(\chi(X_\infty) - \chi(X_0)) - (\chi(X_\infty \cap H) - \chi(X_0 \cap H))\right].
\end{split}	
\end{equation}
Furthermore, it follows from  \cite[Section 10.4]{Max} (see also \cite[Proposition 5.1]{PP} and \cite[Proposition 4.1]{MSS}) that one has the following identity:
\begin{equation}\label{eq6}
	\chi(X_{\infty}) - \chi(X_0)=\chi(\varphi_f (1_{X \setminus X_{\infty}})),
\end{equation}
where $$\varphi_f:CF(X \setminus X_{\infty}) \to CF(X_0 \setminus X_{\infty})$$ denotes the vanishing cycle functor defined on constructible functions, and $1_{X \setminus X_{\infty}}$ is the constant function $1$ on $X \setminus X_{\infty}$.
Similarly, by restricting to the generic (hence smooth) hyperplane section $X^H:=X \cap H$ of $X$, and working with the pencil $X^H_s:=X_s \cap H$ on $X^H$ and the restricted function $f^H:=f|_H$, one gets that:
\begin{equation}\label{eq7}
	\chi(X^H_{\infty}) - \chi(X^H_0)=\chi(\varphi_{f^H} (1_{X^H \setminus X^H_{\infty}})).
\end{equation}
Using the base change isomorphism of \cite[Lemma 4.3.4]{Sch}, we also have that 
\begin{equation}\label{eq8}
\varphi_{f^H} (1_{X^H \setminus X^H_{\infty}}) = \varphi_f (1_{X \setminus X_{\infty}}) |_H.	
\end{equation}
Substituting the identities (\ref{eq6}), (\ref{eq7}), and (\ref{eq8}) in (\ref{eq5}) we get:
\begin{equation}\label{eq9}
\begin{split}	
 \UED(X) - \GED(X) &=
(-1)^{\dim (X)} \left[ 
\chi(\varphi_f (1_{X \setminus X_{\infty}}))-
\chi(\varphi_{f} (1_{X \setminus X_{\infty}})|_H) \right] \\
&= (-1)^{\dim (X)} \chi(\varphi_{f} (1_{X \setminus X_{\infty}})|_{X\setminus(X_{\infty} \cup H)}),
\end{split}
\end{equation}
where the last equality uses the fact that $H$ is generically chosen. %, $(X_0 \setminus X_{\infty})\cap H$ is a union of strata of $X_0 \setminus X_{\infty}$. 
\end{proof}

\begin{remark}
We further note that for generic weight $\bfw$, the constructible function $\varphi_{f} (1_{X \setminus Q_\bfw})|_{X\setminus(Q_\bfw \cup H)}$ is in fact supported on the (complement of $H$ in the) singular locus of the zero-fiber of $f$, i.e., on $\sing(X\cap Q)\setminus (Q_\bfw \cup H)$. %, where $Z:={\rm Sing}(X \cap Q)$ is  the locus of non-transversal intersection of $X$ and $Q$.
\end{remark}

As an immediate consequence of Theorem \ref{th-main}, we get the following result (also proved in \cite[Corollary 6.3]{AH} by using characteristic classes):
\begin{corollary}\label{cor1}
	Under Notation \ref{not1}, assume that $\sing(X\cap Q)$ consists of isolated points. Then:
	\begin{equation}
\DED(X)%=\GED(X) - \UED(X)
=\sum_{x \in {\rm Sing}(X \cap Q)} \mu_x,
\end{equation}
where $\mu_x$ is the Milnor number of the isolated singularity $x \in {\rm Sing}(X \cap Q)$.
\end{corollary}

\begin{remark}\label{ind}
In the statement of Corollary \ref{cor1} we use the fact that the Milnor fibration of a hypersurface singularity germ does not depend on the choice of a local equation for the germ. In particular, at points $x \notin X_0 \cap X_\infty$ one can use freely $f$ in place of $f_0$ (and viceversa) when considering Milnor fibers of such points in $X_0=X \cap Q$.	
\end{remark}

As another important special case, assume that $Z={\rm Sing}(X \cap Q)$ is a closed (smooth and connected) stratum in a Whitney stratification of $X_0=X \cap Q$, that is, $X_0$ is equisingular along $Z$. Then the Milnor fiber of $f_0$ at any point $x \in Z$ has the homotopy type of a bouquet of spheres of dimension $\dim(X_0) -\dim(Z)$, and let us denote by $\mu$ the number of these spheres (this is the transversal Milnor number at a point $x \in Z$, i.e., the Milnor number of the isolated singularity at $x$ in a normal slice to the stratum $Z$). In particular, using Remark \ref{ind}, we get in this case that:
\begin{equation*}
\begin{split}
\chi(\varphi_{f} (1_{X \setminus X_{\infty}})|_{X\setminus(X_{\infty} \cup H)}) &=(-1)^{\dim(X_0)-\dim(Z)}\mu \cdot \chi(Z \setminus (Q_\bfw \cup H))\\
&=(-1)^{\dim(X_0)}\mu \cdot \GED(Z),
\end{split}
\end{equation*}
where the second equality follows from (\ref{eq4}).
Then Theorem \ref{th-main} yields the following:
\begin{corollary}\label{cor:easyZ}
Under Notation \ref{not1}, assume that $Z=\sing(X\cap Q)$ is connected and $X \cap Q$ is equisingular along $Z$. Then:
\begin{equation}\label{th-equi}
\DED(X)%=\GED(X) - \UED(X)
=\mu \cdot \GED(Z),
\end{equation}
where $\mu$ is the Milnor number of the isolated transversal singularity at some point of $Z$.
\end{corollary}

\medskip

For the remaining of this section, we deal with the case when $Z={\rm Sing}(X \cap Q)$ is itself Whitney stratified by arbitrary singularities. 
Choose as before a Whitney stratification $\mathscr{X}$ of the hypersurface $X_0=X \cap Q$ in $X$ so that $Z$ is a union of strata. 
%By Thom's second isotopy lemma, the topological type of the Milnor fibers is constant along the strata of $\mathscr{X}$. 
Recall that any strata $W, V \in \mathscr{X}$ satisfy the frontier condition: $W \cap {\bar V} \neq \emptyset$ implies that $W \subset {\bar V}$. In particular, $\mathscr{X}$ is partially ordered by:
$$W \leq V \ \iff W \subset {\bar V}.$$
We write $W<V$ if $W \leq V$ and $W \neq V$.
By the genericity assumption, the strata of $\mathscr{X}$ are intersected transversally by $H$ and $Q_\bfw$ (for $\bfw$ generic). Let $\mathscr{X}_0$ (with the induced partial order $\leq$) denote the collection of singular strata of $\mathscr{X}$, i.e., strata of $X_0$ which are contained in $Z$. 
%For each $V \in \mathscr{X}_0$, we define \begin{equation}\label{ga} \gamma_V:={\rm Eu}_{\bar V} \vert_{\pp^n \setminus (Q_\bfw \cup H)},\end{equation}
%which is a constructible function supported on ${\bar V} \setminus (Q_\bfw \cup H)$. Then  (\ref{eq4}) yields that $$\GED(\bar V)=(1)^{\dim(V)} \chi(\gamma_V).$$
Recall that the constructible function \begin{equation}\label{al} \alpha:=\varphi_{f} (1_{X \setminus Q_\bfw})|_{X\setminus(Q_\bfw \cup H)}\end{equation} of Theorem \ref{th-main} is supported on $Z \setminus (Q_\bfw \cup H)$. Our goal is to express $\alpha$ in terms of the constructible functions ${\rm Eu}_{\bar V} \vert_{\pp^n \setminus (Q_\bfw \cup H)}$, with $V \in \mathscr{X}_0$. 

We first recall some well-known facts about constructible functions. Let $CF_{\mathscr{X}_0}(Z)$ denote the abelian group of $\mathscr{X}$-constructible functions on $X \cap Q$ which are supported on $Z$. Then we have the following:
%so that its intersections with $Q$, $Q_\bfw$ (for $\bfw$ generic), $H$, and all intersections of these spaces in $X$, are unions of strata of $\mathscr{X}$. 
%Let $\mathscr{V}$ be the collection of strata in $\mathscr{X}$ which are contained in $Z \setminus (Q_\bfw \cup H)$. 
\begin{lemma}\label{lem1}
The collection $\{ {\rm Eu}_{\bar V} \mid V \in \mathscr{X}_0 \}$ is a basis of $CF_{\mathscr{X}_0}(Z)$. %, the abelian group of $\mathscr{X}_0$-constructible functions on $Z$ (or, better said, $\mathscr{X}$-constructible functions on $X \cap Q$ which are supported on $Z$).
\end{lemma}

\begin{proof}
%For any $V \in \mathscr{X}_0$, denote for simplicity:
%$$\beta_V:=1_V\vert_{\pp^n \setminus (Q_\bfw \cup H)}=1_{V \setminus (Q_\bfw \cup H)}.$$
This is well-known. We sketch a proof to set the notations for further use.

Using the distinguished basis $\{ 1_V \mid V \in \mathscr{X}_0 \}$  of $CF_{\mathscr{X}_0}(Z)$, we write 
$${\rm Eu}_{\bar V} =\sum_{W \leq V} a_{W,V} \cdot 1_W,$$
with transition matrix $A=(a_{W,V})$ given by:
$$a_{W,V}={\rm Eu}_{\bar V}(w), \ {\rm for} \ w \in W.$$
By the properties of the local Euler obstruction function, we have that $${\rm Eu}_{\bar V} \vert_V = 1_V,$$
and, for $w \in W$, ${\rm Eu}_{\bar V}(w)\neq 0$ only if $W \leq V$. 
So the transition matrix $A=(a_{W,V})$ is upper-triangular with respect to the partial order $\leq$, with all diagonal entries equal to $1$. In particular, $A$ is invertible, so $\{ {\rm Eu}_{\bar V} \mid V \in \mathscr{X}_0 \}$ is indeed a basis of $CF_{\mathscr{X}_0}(Z)$.
\end{proof}

%\begin{remark}\label{inv}
%It follows from \cite[Proposition 3.6.2]{St} that the non-zero entries of the inverse 	$A^{-1}=(a'_{W,V})$ of the upper-triangular matrix $A=(a_{W,V})$ can inductively be calculated by: \begin{equation}\label{ap1} a'_{V,V}=1\end{equation} and, for $W<V$, the corresponding off-diagonal entry is computed as:
%\begin{equation}\label{ap2} a'_{W,V}=-\sum_{W \leq T <V} a'_{W,T} \cdot a_{T,V}.\end{equation}
%So, 

For any stratum $V \in \mathscr{X}_0$, we can now write $1_V$ in the basis of Lemma \ref{lem1} as:
\begin{equation}\label{eq15}
1_V=\sum_{W \leq V} b_{W,V} \cdot {\rm Eu}_{\bar W},
\end{equation}
where the matrix $B=(b_{W,V})$ is the inverse of the matrix $A=(a_{W,V})$ from the proof of the above lemma. In particular, $B$ is upper triangular with all diagonal entries equal to $1$ and with non-zero off-diagonal entries computed inductively by the inversion formula of \cite[Proposition 3.6.2]{St}. We also note that 
\be
b_{W,V}=(-1)^{\dim W} e_{W,V},
\ee
where $e_{W,V}$ is the {\it Euler obstruction} of the pair of strata $(W,V)$, e.g., see \cite[Section 1.1]{EM}.
With this interpretation, a result of Kashiwara \cite{Ka} (see also \cite[8.2]{Gin}, or \cite[Theorem 1.1]{EM}), states that the non-zero off-diagonal entries of $B$ can be given a topological interpretation in terms of complex links of pairs of strata. Specifically, for strata $W<V$,  one has:
\begin{equation}\label{cl}
b_{W,V}=-\chi_c(L_{W,V}),
\end{equation}
where $\chi_c$ denotes the Euler characteristic of compactly supported cohomology, and $L_{W,V}$ is the {\it complex link} of the pair of strata $W, V$ (that is, the intersection of $V$ with a nearby hyperplane near $W$ and normal to $W$; see, e.g., \cite[Theorem 1.1]{EM} for a precise definition). %{\color{red}This formula is different from the one in ``The ML degree of the mixture of independence models". }
%\end{remark}

In the above notations, we have the following:
\begin{lemma}\label{lem2}
Let $\delta \in CF_{\mathscr{X}_0}(Z)$ be a constructible function, written in terms in the above distinguished bases as:
\begin{equation}\label{del} \delta=\sum_{V \in \mathscr{X}_0} \mu_V \cdot 1_V = \sum_{V \in \mathscr{X}_0} \alpha_V \cdot {\rm Eu}_{\bar V}\end{equation} for some integers $\mu_V, \alpha_V$.
Then for any $W \in \mathscr{X}_0$ one has:
\begin{equation}\label{exp}
\alpha_W=\sum_{\{V \mid W \leq V\}} 	b_{W,V} \cdot \mu_V.
\end{equation}
\end{lemma}
\begin{proof}
Evaluate (\ref{del}) at $w \in W$ to get:
$$\mu_W=\sum_{\{V \mid W \leq V\}} \alpha_V \cdot {\rm Eu}_{\bar V}(w)=\sum_{\{V \mid W \leq V\}} \alpha_V \cdot a_{W,V}.$$	
Then (\ref{exp}) follows since $B=(b_{W,V})$ is the inverse of $A=(a_{W,V})$.
\end{proof}

In order to deal with the function $\alpha=\varphi_{f} (1_{X \setminus Q_\bfw})|_{X\setminus(Q_\bfw \cup H)}$ of Theorem \ref{th-main}, we need to restrict the statements of Lemma \ref{lem1} and Lemma \ref{lem2} to $Z \setminus (Q_\bfw \cup H)$. Using Remark \ref{ind}, the coefficients $\mu_V$ of $\alpha$ in the basis $\{1_{V \setminus (Q_\bfw \cup H)} \mid V \in \mathscr{X}_0\}$ of constructible functions supported on $Z \setminus (Q_\bfw \cup H)$ are given by:
\begin{equation}\label{muv}
\mu_V=	\chi(\widetilde{H}^*(F_{V};\bQ)),
\end{equation}
i.e., the Euler characteristic of the reduced cohomology of the Milnor fiber $F_{V}$ of $f_0$ at some point in $V$. Plugging (\ref{muv}) and (\ref{cl}) into (\ref{exp}), and expressing $\alpha$ in terms of the basis $\{ {\rm Eu}_{\bar V} \vert_{\pp^n \setminus (Q_\bfw \cup H)} \mid V \in \mathscr{X}_0 \}$ of constructible functions with support on $Z \setminus (Q_\bfw \cup H)$, we get by Theorem \ref{th-main} and Remark \ref{ind} the following generalization of (\ref{th-equi}) to arbitrary singularities:
\begin{theorem}\label{thm-gen}
%Let $X \subset \pp^n$ be a smooth irreducible projective variety.
%Let $Z$ denote the singular (i.e., non-transversal intersection) locus of $X\cap Q$, where $Q$ is the isotropic quadric defined by $x_0^2+\cdots+x_n^2$.
%Suppose $X\cap Q$  has a Whitney stratification $\mathscr{X}$, and denote by $\mathscr{X}_0$ the collection of strata contained in $Z$.
%Then,
Under Notation \ref{not1},
\begin{equation}
\DED(X)%=\GED(X)-\UED(X)
=\sum_{V \in \mathscr{X}_0} (-1)^{\codim_{X \cap Q} V} \alpha_V \cdot\GED(\bar V)
\end{equation}
%where each multiplicity $a_V$ defined as: $$a_V:=(-1)^{\codim_{X \cap Q} V} \alpha_V,$$
with %$\alpha_V$ %computed as in (\ref{exp}).
$$\alpha_V=\sum_{\{S \mid V \leq S\}} 	b_{V,S} \cdot \mu_S=\mu_V-\sum_{\{S \mid V < S\}} \chi_c(L_{V,S}) \cdot \mu_S.$$
Here, for any stratum $V \in \mathscr{X}_0$, $\mu_V$ is the Euler characteristic of the reduced cohomology of the Milnor fiber $F_{V}$ of the hypersurface $X \cap Q \subset X$ at some point in $V$, and $L_{V,S}$ denotes the complex link of a pair of distinct strata $(V,S)$ with $V \subset {\bar S}$. 
\end{theorem}

%\pagebreak

\section{Examples}\label{sec:example}

%This section begins with examples of ED degree defects for low dimensional hypersurfaces. 
%Then, we move on to an important case study on restrictions to linear spaces and rank one matrices. Finally, we conclude with examples of ED defects for singular varieties to highlight the difficulties in generalizing the results to this situation. 
%\begin{ex}[Points]
%$(x-2y)(x^2+y^2)$
%\end{ex}

%If $X$ is the variety of rank one matrices as in Example~\ref{} and $\cL$ is a general linear space,  then $\DED(X\cap\cL)=\GED(Z\cap\cL)$ where $Z$ is the singular locus of $X\cap Q$. 

\begin{ex}[$2\times 2$ Determinant]\label{ex:illustrative}
Let $X$ denote the smooth irreducible subvariety of $\pp^3$ defined by $x_0x_3-x_1x_2=0$, %and  let $Q_\bfw$ denote the isotropic quadric defined by $\{w_0x_0^2+\cdots + w_3x_3^2=0\}$ with $\bfw=(w_0,\dots,w_3)$.
and  let $Q$ 
denote the isotropic quadric $\{(x_0:\dots:x_3)\in\pp^n\mid \sum_{i=0}^3 x_i^2=0\}$,
%denote the isotropic quadric defined by $x_0^2+\cdots  + x_3^2=0$.
The variety $X \cap Q$ consists of four lines and has precisely four isolated singularities.
This is illustrated in Figure~\ref{fig:rankOne} 
where we restrict $X$ to an affine chart by setting $x_0=1$  and make a change of coordinates to plot the figures effectively. 
By Corollary~\ref{cor1}, we have 
\[
\DED(X)=\sum_{x \in {\rm Sing}(X \cap Q)} \mu_x=1+1+1+1 %\sum_{x \in {\rm Sing}(X \cap Q)} 1,
\]
where $\mu_x=1$ is the Milnor number of the isolated singularity $x \in {\rm Sing}(X \cap Q)$.
This agrees with computations from \cite{DHOST}, as  $\GED(X)=6$ (cf. \cite{DHOST}[Example 7.11] and  $\UED(X)=2$ (cf. \cite{DHOST}[Example 2.4]).
Furthermore, it is much easier to compute $\DED(X)$ directly rather than computing the 
two Euler characteristics in \eqref{eq2} and \eqref{eq4} separately.

 \begin{figure}[htb!]
   \label{fig:rankOne}
 \centering
   \begin{picture}(153,183)
     {\includegraphics[scale=.6]{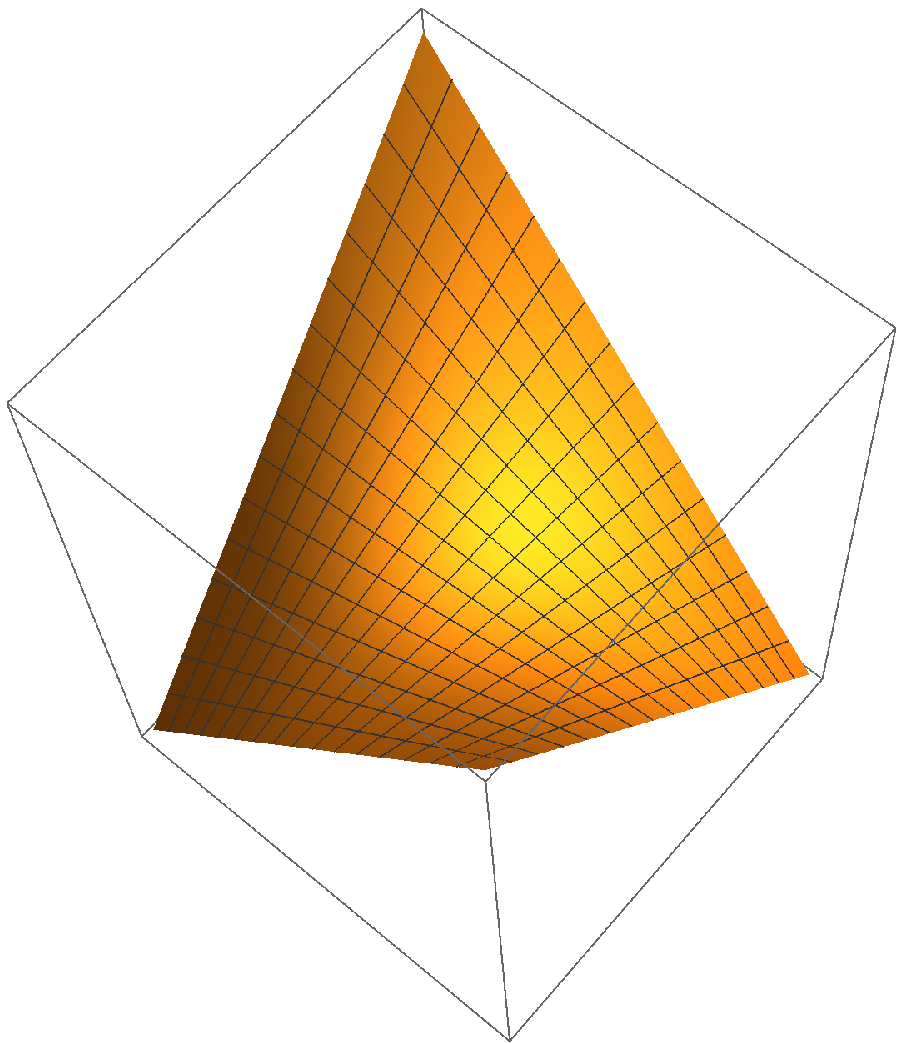}}
	\put(-150,-20){$X$ up to coordinate change}
   \end{picture}
   \qquad\qquad
   \begin{picture}(97,183)(16,0)
     {\includegraphics[scale=.6]{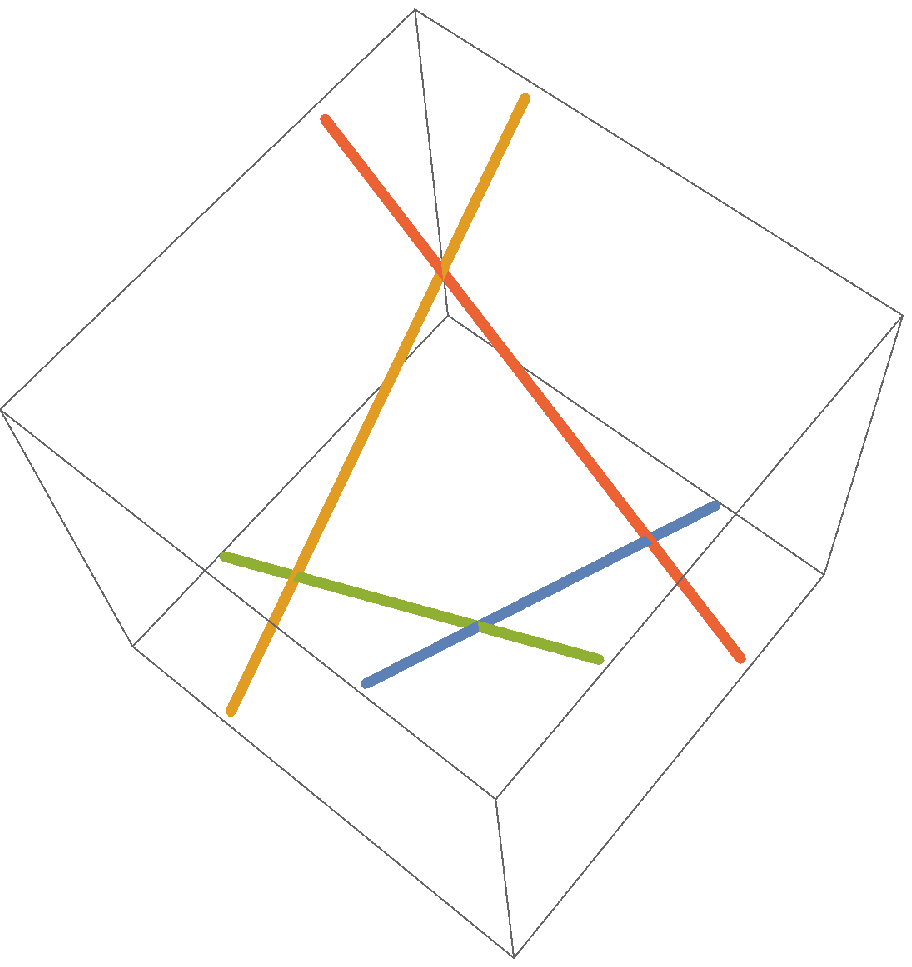}}
	\put(-150,-20){$X\cap Q$ up to coordinate change}
   \end{picture}
	\newline
   \caption{\texttt{Left}: We illustrate $X$ by plotting  
   $\{(\sqrt{-1}x_1,\sqrt{-1}x_2,x_3)\mid (1 : x_1 : x_2 : x_3)\in X\}$, and we see that it is smooth. 
   \texttt{Right}: We illustrate $X\cap Q$  by plotting $\{(\sqrt{-1}x_1,\sqrt{-1}x_2,x_3)\mid (1:x_1 : x_2 : x_3)\in X\cap Q\}$, and we see the four points where any two lines meet correspond to the four points in~$Z$. % \blue{The plots need to be redone eventually.}
   }
\end{figure}

\end{ex}

{
\begin{ex}[Kinetic Proofreading Networks: McKeithan Model]
The following example is motivated by chemical reaction networks and was initially proposed by 
McKeithan~\cite{McKeithan5042}.  
%T.W. McKeithan,Kinetic proofreading in T-cell receptor signal transduction,PNAS 92 (1995),5042-5046.
We follow the formulation from~\cite{AH2018}[Section 3.3].
%``The model we present in this paper was initially proposed by McKeithan [42] to understand the simultaneous high sensitivity and high selectivity of antigen recognition in T cells."

The \emph{affine N-site McKeithan Variety} is an affine toric variety given by the image of the map
\[
\cc^2 \to \cc^{N+2},\quad
(r,s)\mapsto(rs,rs,\dots,rs,r,s)=(x_1,\dots,x_N,a,b).
\]
The set of implicit equations of the projective closure $X_N$ that is obtained by homogenizing with respect to $x_0$ is 
\[
\{
x_1=x_2=\dots = x_N,\quad
x_0x_N=ab\}.
\]
When $N=1$, this specializes to the $2\times 2$ determinant in our previous example. 
The results of \cite{AH2018}, imply $\GED(X_N)=6$ and $\DED(X_N)=0$ for $N>1$.
On the other hand,  
the set of implicit equations of the projective closure $Y_N$ that is obtained by homogenizing with respect to $\sqrt{N} x_0$  is
\[
\{x_1=x_2=\dots = x_N,\quad
\sqrt{N} x_0x_N-ab\}.
\]
While  the generic Euclidean distance degrees of $X_N$ and $Y_N$ coincide for all $N$,
 the unit Euclidean distance degrees can be different. 
In fact, it follows as in  the previous example that
$\DED(Y_N)=4$ for all $N \geq 1$. To see this, note the intersection of $Y_n$ with the isotropic quadric consists of four line intersecting at four points 
%(the Milnor number of each isolated singularity of this intersection is also one) 
as in Figure~\ref{fig:rankOne} but in a higher dimensional ambient space. 
\end{ex}

}

\begin{ex}[Rank one matrices]

The variety $X$ in Example~\ref{ex:illustrative} consists of $2\times2$ matrices with rank equal to one. More generally, let $X=X_{s,t}$ denote the subvariety of $\pp^{st-1}$ defined by the $2\times 2$ minors of the matrix 
\[
\begin{bmatrix}
x_{1,1} & \dots & x_{1,t}\\
\vdots & \dots & \vdots\\
x_{s,1} & \dots & x_{s,t}\\
\end{bmatrix}.
\]

The variety $X$ is  smooth and irreducible. In fact, $X$ is the image of the Segre embedding $\sigma: \bP^{s-1}\times \bP^{t-1}\to \bP^{st-1}$. Instead of studying the intersection $X\cap Q$ in $\bP^{st-1}$, we study the isomorphic variety $\sigma^{-1}(Q)$ in $\bP^{s-1}\times \bP^{t-1}$. Let $y_1, \ldots, y_s$ and $z_1, \ldots, z_t$ be the homogeneous coordinates of $\bP^{s-1}$ and $\bP^{t-1}$ respectively. Since the isotropic quadric $Q\subset\pp^{st-1}$ is defined by $\sum_{i,j}x^2_{ij}=0$, the preimage $\sigma^{-1}(Q)\subset \bP^{s-1}\times \bP^{t-1}$ is defined by $\sum_{i, j}(y_i z_j)^2=0$. 
%\magenta{Do we like pure-dimensional, pure dimensional, equi-dimensional, or equidimensional?}.
%The isotropic quadric $Q\subset\pp^{st-1}$ is defined by
% $\sum_{0\leq i\leq s-1,0\leq j\leq t-1}{x_{i,j}}$
%$\sum_{i,j}x^2_{ij}=0$, and we let $Z$ denote the singular locus of $X\cap Q$.
Notice that $\sum_{i, j}(y_i z_j)^2=\big(\sum_{i}y_i^2\big)\cdot \big(\sum_j z_j^2\big)$. Thus, $\sigma^{-1}(Q)$ consists of two smooth irreducible components,
\[Z_1\coloneqq\big\{[y_i]\in\bP^{s-1}\large\mid \sum_{i}y_i^2=0\big\}\times \bP^{t-1}\]
and
\[Z_2\coloneqq\bP^{s-1}\times \big\{[z_j]\in\bP^{t-1}\large\mid \sum_{j}z_j^2=0\big\}.\]
Clearly, $Z_1$ intersects $Z_2$ transversally and $\sigma^{-1}(Q)$ is equisingular along $Z_1\cap Z_2$. Therefore, $\sigma^{-1}(Z)=Z_1\cap Z_2$ or equivalently $Z=\sigma(Z_1\cap Z_2)$. Take a point $P\in Z_1\cap Z_2$ and take a two-dimensional general slice $V\subset \bP^{s-1}\times \bP^{t-1}$ passing through $P$. Near $P$, $V\cap Z_1$ and $V\cap Z_2$ are two smooth curves intersecting transversally at $P$. It is well-known that the Milnor number of a nodal curve singularity is 1. Therefore, by Corollary~\ref{cor:easyZ}, 
%It follows  $X \cap Q$ is equisingular along $Z$
%and the {Milnor number of the isolated transversal singularity at a generic point of $Z$ is one}.  
%\magenta{More details here on why it equals one or is this suppose to be obvious?}
%
\[
\DED(X)=\mu \cdot \GED(Z)=1\cdot \GED(Z)=(-1)^{\dim Z}\chi(Z\setminus (Q_{\bfw}\cup H))
%\blue{Insert something here?}.
\]%
where $\bfw$ is a generic weight and $H$ a general hyperplane. The last term of the above formula can be computed as follows. 
\iffalse
\blue{
Since Euler characteristic is additive on subvarieties, we have
\begin{equation}\label{eq_inex}
\chi(Z\setminus (Q_{\bfw}\cup H))=\chi(Z)-\chi(Z\cap Q_{\bfw})-\chi(Z\cap H)+\chi(Z\cap Q_{\bfw}\cap H).
\end{equation}
Notice that $Z\cong \sigma^{-1}(Z)=Z_1\cap Z_2$. Thus, the right hand side of equation (\ref{eq_inex}) is equal to
\begin{equation}\label{eq_inex2}
%\chi(Z\setminus (Q_{\bfw}\cup H))=
\chi(Z_1\cap Z_2)-\chi(Z_1\cap Z_2\cap \sigma^{-1}(Q_{\bfw}))-\chi(Z_1\cap Z_2\cap \sigma^{-1}(H))+\chi(Z_1\cap Z_2\cap \sigma^{-1}(Q_{\bfw})\cap \sigma^{-1}(H)).
\end{equation}
}
\fi
% I reformatted this.
Since the Euler characteristic is additive on subvarieties, we have
\[
\chi(Z\setminus (Q_{\bfw}\cup H))=\chi(Z)-\chi(Z\cap Q_{\bfw})-\chi(Z\cap H)+\chi(Z\cap Q_{\bfw}\cap H).
\]
Furthermore, 
since $Z\cong \sigma^{-1}(Z)=Z_1\cap Z_2$,
we have
\begin{align*}
\chi(Z)  				&  =\chi(Z_1\cap Z_2)\\
 \chi(Z\cap Q_{\bfw}) 	&=\chi(Z_1\cap Z_2\cap \sigma^{-1}(Q_{\bfw}))\\
\chi(Z\cap H)			& =\chi(Z_1\cap Z_2\cap \sigma^{-1}(H))\\
\chi(Z\cap Q_{\bfw}\cap H)& =\chi(Z_1\cap Z_2\cap \sigma^{-1}(Q_{\bfw})\cap \sigma^{-1}(H)).
\end{align*}
All the intersections in the above equations are smooth, hence each right hand side can be computed using Chern classes, see e.g. \cite[Page 15]{MRW2018}. 

Notice that $\sigma^{-1}(Q_{\bfw})\subset \bP^{s-1}\times \bP^{t-1}$ is a hypersurface of bidegree $(2, 2)$ and $\sigma^{-1}(H)\subset \bP^{s-1}\times \bP^{t-1}$ is a hypersurface of bidegree $(1, 1)$.  Thus, the values of $\chi(Z_1\cap Z_2)$, $\chi(Z_1\cap Z_2\cap \sigma^{-1}(Q_{\bfw}))$, $\chi(Z_1\cap Z_2\cap \sigma^{-1}(H))$ and $\chi(Z_1\cap Z_2\cap \sigma^{-1}(Q_{\bfw})\cap \sigma^{-1}(H))$ is equal to the coefficient of $[H_1]^{s-1}[H_2]^{t-1}$ in the following power series, respectively,
\newcommand{\EQFone}{2[H_1]\cdot 2[H_2]\cdot\frac{  (1+[H_1])^s (1+[H_2])^t}{(1+2[H_1])(1+2[H_2])}}
\begin{align}
 \label{eq_f1} \EQFone; &\\
\label{eq_f2} \EQFone & \cdot \frac{ (2[H_1]+2[H_2])}{(1+2[H_1]+2[H_2])}; \\
\label{eq_f3} \EQFone & \cdot \frac{  ([H_1] + [H_2]) }{(1+[H_1]+[H_2])};\\
\label{eq_f4}\EQFone & \cdot \frac{ (2[H_1]+2[H_2])( [H_1] + [H_2])}{(1+2[H_1]+2[H_2])(1+[H_1]+[H_2])};
\end{align}
where $[H_1]$ and $[H_2]$ are considered as formal variables.
Thus, we have  $\DED(X)$ is equal to the coefficient of $[H_1]^{s-1}[H_2]^{t-1}$
in %\eqref{eq_f1}-\eqref{eq_f2}-\eqref{eq_f3}+\eqref{eq_f4}, which is equal to
%i.e.,  the coefficient of  $[H_1]^{s-1}[H_2]^{t-1}$ in
\begin{equation}\label{eq:powerSeries}
\EQFone\cdot 
\frac{1}{(1+2[H_1]+2[H_2])(1+[H_1]+[H_2])}.
\end{equation}
After expanding the factors that do not depend on $s,t$, we see $\eqref{eq:powerSeries}$ equals 
\[
(1+[H_1])^s (1+[H_2])^t\cdot\sum_{0\leq i, j} c_{i,j} [H_1]^i[H_2]^j
\]
where the $c_{i,j}$ are integers that do not depend on $s,t$. 
Thus, for any $s,t,$
\[
\DED(X)=\sum_{k=0}^{s-1}\sum_{\ell=0}^{t-1}\binom{s}{k}\binom{t}{\ell}c_{s-1-k,t-1-\ell}.
\]

\begin{remark}
A generalization of the above setup to rank-one tensors is considered in \cite[Example 9.6]{AH}, where the $\UED$ is computed using a similar method. 
\end{remark}

%{\color{red}Botong can finish the rest of this example. }
%We now explain how to determine $\GED(Z)$ and the last equality in the previous equation. 
%\blue{Does this look correct?}
%The variety $X$ is parameterized by $\pp^{s-1}\times\pp^{t-1}$. 
%Let $Q_s\subset \pp^{s-1}$ and $Q_t\subset \pp^{t-1}$ denote isotropic quadrics. 
%The variety $X\cap Q$ has two components, which are parameterized by 
%$Q_{s-1}\times \pp^{t-1}$ and $\pp^{s-1}\times Q_{t-1}$ where $Q_n$ is the subvariety of $\pp^n$ defined by $x_0^2+\cdots+x_n^2.$  Note these two components are each irreducible when $s,t > 2$.
%The intersection of these two components is the singular locus $Z$ of $X\cap Q$. \blue{What else to say?}

%Corollary: Reinterpreting Duality Conjecture {\color{red}Botong doesn't think we know how to prove the duality property topologically yet. }

When combining the above example with Corollary~\ref{cor-slice}, we get the following:

\begin{corollary}
If $X$ is the variety of rank one matrices as in Example~\ref{ex:illustrative} and $\cL$ is a general linear space, 
then $\DED(X\cap\cL)=\GED(Z\cap\cL)$ where $Z$ is the singular locus of $X\cap Q$. %{\color{red}It will be too complicated to explicitly compute $\GED(Z\cap\cL)$.}
\end{corollary}

\end{ex}

\newcommand{\jj}{\sqrt{-1}}

\begin{ex}[Quadric surface]
Consider the hypersurface $X$defined by
\[
f=(x_1-\jj x_0)^2+2(x_3-\jj x_2 )^2+q\]
where 
 $q=x_0^2+x_1^2+x_2^2+x_3^2$.
The variety $X\cap Q$ consists of three lines. Two of the lines $L_1,L_2$ are generically reduced, but one of the lines $L_3$ is with multiplicity two.  
This explains why the degree of the ideal $\langle f,q\rangle$ is four. 
%\blue{Say something about complex links being points.}
The radical ideal of $L_3$ is 
$\langle x_2 + \jj x_3, x_1  - \jj x_0\rangle$ and defines $Z$.  
The union of lines $L_1\cup L_2$ intersect $L_3$ at two distinct points:
$\{P_1,P_2\}$.
%$\{P_1,P_2\}=\{(-\jj:1:\eta:\jj \eta) \mid \eta^2=2\jj\}$.

%\blue{Fill in correct numbers.}
We stratify $Z$ by $S_0=Z\setminus \{P_1,P_2\}$ and $S_i=\{P_i\}$ for $i=1,2$. For $i=1, 2$, the complex link $L_{S_i, S_0}$ consists of a point. 
According to Equation~\eqref{exp}:
\[
\begin{bmatrix}
\alpha_{ P_2}\\
\alpha_{ P_1}\\
\alpha_{ S_0}
\end{bmatrix}= 
% [e_{W,V}
\begin{bmatrix}
1 & 0 & -1\\
 0 & 1 & -1\\
 0 & 0 & 1
\end{bmatrix}
\begin{bmatrix}
  \mu_{P_2}\\
  \mu_{P_1}\\
  \mu_{S_0}
\end{bmatrix}.
\]

%$${\rm Eu}_{\bar V} =\sum_{W \leq V} a_{W,V} \cdot 1_W,$$
%$$1_V=\sum_{W \leq V} e_{W,V} \cdot {\rm Eu}_{\bar W}$$

The Milnor fiber $F_{S_0}$ is homotopy equivalent to $\{x^2=1\}\subset \bC^2$, and for $i=1, 2$ the Milnor fiber $F_{S_i}$ is homotopy equivalent to $\{x^2y=1\} \simeq \bC^*$. Therefore, $\chi(F_{S_0})=2$ and $\chi(F_{S_i})=0$ for $i=1, 2$. %Since the Milnor fiber $\{x^2y=1\}\subset \bC^2$ admits a transitive $\bC^*$-action, its Euler characteristic is equal to zero. 
According to Theorem~\ref{thm-gen},
%\blue{Jose needs to check the signs.} {\color{red}Definitely!}
\[
\begin{split}	
\DED(X) % &=\GED(X)-\UED(X)\\
&=
  \alpha_{S_0} \cdot\GED(S_0)   -  \alpha_{P_1} \cdot\GED(P_1)-  \alpha_{P_2} \cdot\GED(P_2)\\
  &=
  \alpha_{S_0} \cdot 1   - \alpha_{P_1} \cdot 1-  \alpha_{P_2} \cdot 1\\
  &=  \mu_{S_0} \cdot 1   -  (\mu_{P_1}-\mu_{S_0}) \cdot 1-   (\mu_{P_2}-\mu_{S_0}) \cdot 1\\
  &=  1 \cdot 1   -  (-1-1) -(-1-1)\\
  &=5.
  \end{split}
\]

\begin{remark}
These topological computations agree with the fact  that
$\GED(X)=6$ and
$\UED(X)=1$. 
We computed these numbers  
using our
\texttt{Macaulay2} \cite{M2} package 
$\texttt{EuclideanDistanceDegree}$, which is  available at
\begin{quote}
\url{https://github.com/JoseMath/EuclideanDistanceDegree/} 
\end{quote}
This package implements  
Grobner basis methods
and continuation methods (specifically, we used \texttt{Bertini}  \cite{Bertini4M2,bertinibook}).
\end{remark}
\begin{quote}
\begin{verbatim}
-* Macaulay2 code to compute EDdefect(V(F)) *-
i1 : loadPackage"EuclideanDistanceDegree";
i2 : kk=QQ[I]/ideal(I^2+1);
i3 : T=kk[x0,x1,x2,x3];
i4 : q=x0^2+x1^2+x2^2+x3^2;
i5 : F={(x1-I*x0)^2+2*(x3-I*x2)^2+q};
 --Symbolic computation (Grobner bases method):
i6 : EDDefect=(determinantalGenericEuclideanDistanceDegree F-
  determinantalUnitEuclideanDistanceDegree F)/(degree kk) 
o6 = 5
 --Numerical computation (Continuation method):
 ----Note: Bertini needs to be installed for this to work.
 --  (i7-i10) Create directories and write Bertini files 
 --  (i11) Run Bertini and computes EDdefect(V(F))
i7 : (dir1,dir2)=(temporaryFileName(),temporaryFileName()); 
i8 : {dir1,dir2}/mkdir; 
i9 : leftKernelGenericEDDegree(dir1,F);
i10 : leftKernelUnitEDDegree(dir2,F);
i11 : EDDefect=runBertiniEDDegree(dir1)-runBertiniEDDegree(dir2)
o11 = 5
\end{verbatim}
\end{quote}
\iffalse
If we deform generic weights to unit weights then the 18 critical points follow paths to 18 points in projective space.
6 of these paths lead to critical points of the usual distance function on $V(F)$
9 paths go to 3 distinct points with multiplicity 3. I would guess these points are P1,P2,P3 from above but haven't checked this.
There are three more paths which I am not so sure about:
2 paths come together somewhere on the line L
The last path, I am not sure where it goes.
\fi
\end{ex}

%\subsection{Rank one matrices}
\iffalse
\section{Limit points}
\blue{I put a place holder for this section. I think we should be able to something here for smooth $X$. 
If nothing else we can put an example here for when $Z$ consists of isolated points and refer to this as future work.   
}
\fi

\bibliographystyle{abbrv}
\bibliography{REF_ED_Degree}
\end{document}